\documentclass{elsarticle}
\usepackage{amsmath, amsthm, amssymb}

\usepackage{tikz}
\usetikzlibrary{matrix,trees,arrows}

\usepackage[utf8]{inputenc}
\usepackage[T1]{fontenc}
\usepackage[mathcal]{euscript}

\usepackage{hyperref}

\newcommand{\bbN}{\mathbb{N}}          
\newcommand{\sym}{\mathcal{S}}         
\newcommand{\s}{S}                     
\newcommand{\symnk}[2]{\sym_{#1}^{#2}} 
\newcommand{\snk}[2]{S_{#1}^{#2}}      
\renewcommand{\P}{\mathcal{P}}         
\newcommand{\Q}{\mathcal{Q}}           
\newcommand{\A}{\mathcal{A}}           
\newcommand{\B}{\mathcal{B}}           
\newcommand{\C}{\mathcal{C}}           
\DeclareMathOperator{\inv}{\text{\sc inv}}

\newcommand{\g}[1]{\node[fill=gray!40]{#1};}

\newtheorem{theorem}{Theorem}
\newtheorem{lemma}[theorem]{Lemma}
\newtheorem{proposition}[theorem]{Proposition}
\newtheorem{corollary}[theorem]{Corollary}
\newtheorem{conjecture}[theorem]{Conjecture}

\theoremstyle{definition}
\newtheorem{definition}[theorem]{Definition}

\begin{document}

\title{Upper bounds for the Stanley-Wilf limit of 1324 and other layered
  patterns} 

\author[strath]{Anders Claesson\fnref{ice}}
\ead{anders.claesson@cis.strath.ac.uk}

\author[prague]{V\'{\i}t Jel\'{\i}nek\fnref{msm}} 
\ead{jelinek@kam.mff.cuni.cz}

\author[strath]{Einar Steingr\'{\i}msson\fnref{ice}}
\ead{Einar.Steingrimsson@cis.strath.ac.uk}

\address[strath]{Department of Computer and Information Sciences, University
  of Strathclyde, Glasgow G1 1XH, UK.}  

\address[prague]{Department of Applied Mathematics, Faculty
of Mathematics and Physics, Charles University, Malostranské náměstí
  25, Prague 1, 118 00, Czech Republic.}

\fntext[ice]{Supported by grant no.\ 090038013 from the
  Icelandic Research Fund.}

\fntext[msm]{Supported by project MSM0021620838 of the Czech
Ministry of Education.}


\begin{abstract}
  We prove that the Stanley-Wilf limit of any layered permutation
  pattern of length $\ell$ is at most~$4\ell^2$, and that the
  Stanley-Wilf limit of the pattern 1324 is at most 16. These bounds
  follow from a more general result showing that a permutation
  avoiding a pattern of a special form is a merge of two permutations,
  each of which avoids a smaller pattern.  If the conjecture is true
  that the maximum Stanley-Wilf limit for patterns of length $\ell$ is
  attained by a layered pattern then this implies an upper bound of
  $4\ell^2$ for the Stanley-Wilf limit of any pattern of length
  $\ell$.

  We also conjecture that, for any $k\ge 0$, the set of 1324-avoiding
  permutations with $k$ inversions contains at least as many
  permutations of length $n+1$ as those of length~$n$. We show that if
  this is true then the Stanley-Wilf limit for 1324 is at
  most~$e^{\pi\sqrt{2/3}} \simeq 13.001954$.
\end{abstract}

\begin{keyword}
 Stanley-Wilf limit \sep layered pattern
\end{keyword}

\maketitle

\section{Introduction}
\label{sec:introduction}

For a permutation pattern $\tau$, let $\sym_n(\tau)$ be the set of
permutations of length $n$ avoiding $\tau$, and let $\s_n(\tau)$ be
the cardinality of $\sym_n(\tau)$.  In 2004, Marcus and
Tardos~\cite{MaTa} proved the Stanley-Wilf conjecture, stating that,
for any pattern $\tau$, $\s_n(\tau)<C^n$ for some constant $C$
depending only on~$\tau$. The limit
$$
L(\tau) = \lim_{n\to\infty}{\s_n(\tau)^{1/n}}
$$ is called the \emph{Stanley-Wilf limit} for
$\tau$. Arratia~\cite{Arratia} has shown that this limit exists for
any pattern $\tau$.

Marcus and Tardos's original proof gives a general upper bound for the
Stanley-Wilf limit of the form
\begin{align*}
  L(\tau) &\leq 15^{2\ell^4\binom{\ell^2}{\ell}},
  \intertext{where $\ell=|\tau|$. This bound was later improved
    by Cibulka~\cite{Cibulka} to}
  L(\tau) &\leq 2^{O(\ell\log\ell)}.
\end{align*}
A result of Valtr presented in~\cite{KaKl} shows that for any pattern
$\tau$ of length $\ell$ we have $L(\tau)\ge (1-o(1))\ell^2/e^3$ as
$\ell\to\infty$.

For certain families of patterns, more precise estimates are
available. An important example is given by the \emph{layered
  patterns}. A permutation $\tau$ is {layered} if it is a
concatenation of decreasing sequences, the letters of each sequence
being smaller than the letters in the following sequences. An example
is 321465798, whose layers are 321, 4, 65, 7, and 98.
B\'ona~\cite{bona-sharper,bona-tenacious,bona-layered-beat} has shown
that $(\ell-1)^2\le L(\tau)\le 2^{O(\ell)}$ for any layered pattern
$\tau$ of length~$\ell$.

A motivation for the study of Stanley-Wilf limits of layered patterns
stems from the fact that these patterns appear to yield the largest
values of $\s_n(\tau)$ among the patterns $\tau$ of a given length.
More precisely, computer enumeration of $\s_n(\tau)$ for patterns
$\tau$ of fixed size up to eight and small $n$ suggests that
$\s_n(\tau)$ is maximized by a layered pattern~$\tau$.  This supports
the following conjecture.

\begin{conjecture}[See \cite{bona-records}]\label{conj-layered}
  Among the patterns of a given length, the largest Stanley-Wilf limit
  is attained by a layered pattern.
\end{conjecture}

We remark that B\'ona, just after Theorem~4.6 in \cite{bona-records}, presents a
stronger version of Conjecture~\ref{conj-layered} as a `long-standing
conjecture'.  The stronger conjecture states that the maximum for
$L(\tau)$ over all $\tau$ of a given length $\ell$ is attained by
$1\oplus21\oplus\cdots\oplus21\oplus1$ or
$1\oplus21\oplus\cdots\oplus21$, depending on whether
$\ell$ is even or odd.

Two patterns $\tau$ and $\sigma$ are \emph{Wilf equivalent} if
$\s_n(\tau)=\s_n(\sigma)$ for all $n$.  If $\tau$ is of length three,
then $\s_n(\tau)$ is the $n$th Catalan number, and so $L(\tau)=4$.
For patterns of length four there are three Wilf (equivalence)
classes, represented by 1234, 1342 and 1324. Regev~\cite{Regev} proved
that $L(1234)=9$ and, more generally, that $L(12\dotsb
\ell)=(\ell-1)^2$. B\'ona \cite{bona-1342} proved that $L(1342)=8$.
In fact, exact formulas for $\s_n(1234)$ and $\s_n(1342)$ are known,
the first one being a special case of such a result for the increasing
pattern of any length, established by Gessel \cite{gessel-symmetric},
the second one obtained by B\'ona~\cite{bona-1342}.

The last Wilf class of patterns of length 4, represented by the pattern 1324, has so far resisted all attempts at exact
enumeration or exact asymptotic formulas. A lower bound was found by
Albert et al.~\cite{albert-elder}, who showed that
$\s_n(1324)>9.47^n$. In particular, that disproved a conjecture of
Arratia \cite{Arratia} that for any pattern $\tau$ of length $\ell$,
$L(\tau)$ is at most $(\ell-1)^2$.

As far as we know, the best published upper bound so far for $L(1324)$
is 288, proved by B\'ona~\cite{bona-tenacious}\footnote{In an earlier
  version of the proof~\cite{bona-wrong}, Bóna claims that $L(1324)\le
  36$, but the argument appears flawed.}.

In this paper, we present a general method that allows us to bound the
Stanley-Wilf limit of an arbitrary layered pattern, and which may also
be used for non-layered patterns of a special form. In particular, for
an arbitrary layered pattern $\tau$ of length~$\ell$, we prove the
bound $L(\tau)\le 4\ell^2$, improving B\'ona's bound of~$2^{O(\ell)}$.
Our bound is sharp up to a multiplicative constant, since
$(\ell-1)^2\le L(\tau)$. If Conjecture~\ref{conj-layered} holds, then this
result has the following consequence, which we state as a separate conjecture.

\begin{conjecture}
  For any pattern $\tau$ of length $\ell$, we have $L(\tau)\le
  4\ell^2$.
\end{conjecture}

For some specific patterns, we are able to give better estimates. Notably, we
are able to show that $L(1324)\leq 16$. These results appear in
Section~\ref{sec:stanley-wilf-limit}.

In Section~\ref{sec:conjectures}, we investigate an approach that may
lead to a further improvement of our bounds, based on the analysis of
pattern-avoiding permutations with a restricted number of
inversions. For a pattern $\tau$, let $\symnk{n}{k}(\tau)$ be the set
of 1324-avoiding permutations of length $n$ with exactly $k$
inversions, and let $\snk{n}{k}(\tau)$ be its cardinality. We
conjecture that for every $n$ and~$k$, we have
$\snk{n}{k}(1324)\le\snk{n+1}{k}(1324)$. We prove that if this
conjecture holds then
$$L(1324)\leq e^{\pi\sqrt{2/3}}\simeq 13.001954.
$$

In the last section, we extend our considerations to more general
patterns. We conjecture that the inequality
$\snk{n}{k}(\tau)\le\snk{n+1}{k}(\tau)$ is valid for any pattern
$\tau$ other than the increasing patterns. As an indirect support of
this conjecture, we describe how the asymptotic behavior of
$\snk{n}{k}(\tau)$ for $k$ fixed and $n$ going to infinity depends on
the structure of~$\tau$.

\section{The Stanley-Wilf limit of 1324 is at most 16}
\label{sec:stanley-wilf-limit}

Let us begin by recalling some standard notions related to permutation
patterns.  Two sequences of integers $a_1\dotsb a_k$ and $b_1\dotsb
b_k$ are \emph{order-isomorphic} if for every $i,j\in\{1,\dotsc,k\}$
we have $a_i<a_j \Leftrightarrow b_i<b_j$. We let $\sym_n$ be the set
of permutations of the letters $\{1,2,\ldots,n\}$. For a permutation
$\pi\in\sym_n$ and a set
$I=\{i_1<i_2<\dotsb<i_k\}\subseteq\{1,\dotsc,n\}$, we let $\pi[I]$
denote the permutation in $\sym_k$ order-isomorphic to the sequence
$\pi(i_1)\pi(i_2)\dotsb \pi(i_k)$. 
A permutation $\pi\in\sym_n$ \emph{contains} a permutation
$\sigma\in\sym_k$ if $\pi[I]=\sigma$ for some~$I$. In such context, $\sigma$ is
often called \emph{a pattern}. 

We say that a permutation $\pi\in\sym_n$ is a \emph{merge} of two permutations
$\sigma\in\sym_k$ and $\tau\in\sym_{n-k}$ if there are two disjoint
sets $I$ and $J$ such that $I\cup J=\{1,\dotsc,n\}$, $\pi[I]=\sigma$
and $\pi[J]=\tau$.  For example, 3175624 is a merge of $\sigma=132$
and $\tau=1423$ with $I=\{1,3,4\}$ and $J=\{2,5,6,7\}$.

For a pair of permutations $\sigma\in\sym_k$ and $\tau\in\sym_{\ell}$,
their \emph{direct sum}, denoted by $\sigma\oplus\tau$, and their
\emph{skew sum}, denoted by $\sigma\ominus\tau$, are defined by
$$
(\sigma\oplus\tau)(i) =
\begin{cases}
  \sigma(i) &\text{if }i\leq k\\
  \tau(i-k)+k &\text{if }i > k
\end{cases}
\quad\,\text{ and }\quad\,
(\sigma\ominus\tau)(i) =
\begin{cases}
  \sigma(i)+\ell &\text{if } i\leq k\\
  \tau(i-k)        &\text{if } i > k.
\end{cases}
$$
For example, $231\oplus3142=2316475$ and $231\ominus3142=6753142$.  A
permutation is \emph{decomposable} if it can be written as a direct
sum of two nonempty permutations, otherwise it is
\emph{indecomposable}. Every permutation $\pi$ can be uniquely written
as a direct sum (possibly with a single summand) of the form
$\pi=\alpha_1\oplus\dotsb\oplus\alpha_m$, where each summand
$\alpha_i$ is indecomposable. The summands $\alpha_1,\dotsc,\alpha_m$
are the \emph{components} of~$\pi$. For example, the permutation
31425786 is decomposed as $31425786=3142\oplus1\oplus231$, which means
that it has three components, corresponding to 3142, 5 and 786.

The key tool in our approach is the next lemma, which shows that a
permutation avoiding a pattern of a particular kind is a merge of two
permutations, each of them avoiding a smaller pattern.

\begin{lemma}\label{lemma:red-blue}
  Let $\sigma$, $\tau$, and $\rho$ be three (possibly empty)
  permutations. Then every permutation avoiding
  $\sigma\oplus(\tau\ominus 1)\oplus\rho$ is a merge of a permutation
  avoiding $\sigma\oplus(\tau\ominus 1)$ and a permutation avoiding
  $(\tau\ominus 1)\oplus\rho$.
\end{lemma}

\begin{proof} We may assume that $\sigma$ and $\rho$ are nonempty, otherwise the lemma holds trivially.
  Let $\pi=\pi_1\dotsb \pi_n$ be a permutation.  Successively color the
  $\pi_i$, in the order $\pi_1,\pi_2,\dots,\pi_n$, red or blue according to the
  following rule:
  \begin{quotation}
    If coloring $\pi_i$ red completes a red occurrence of
    $\sigma\oplus(\tau\ominus 1)$, or if there already is a blue
    element smaller than $\pi_i$, then color $\pi_i$ blue; otherwise
    color $\pi_i$ red.
  \end{quotation}
  Note that the first element, $\pi_1$, will always be colored
  red. Further, the red elements clearly avoid
  $\sigma\oplus(\tau\ominus 1)$. We claim that if $\pi$ avoids
  $\sigma\oplus(\tau\ominus 1)\oplus\rho$ then the blue elements avoid
  $(\tau\ominus 1)\oplus\rho$, and we proceed by proving the
  contrapositive statement. Assume that there is a blue occurrence of
  $(\tau\ominus 1)\oplus\rho$. Let $\tau_B$, $1_B$, and $\rho_B$ be
  the three sets of blue elements corresponding to the three parts
  $\tau$, $1$ and $\rho$ forming the occurrence.  In particular, $1_B$
  contains a single element, which will be denoted by~$\pi_t$.

  Fix a blue element $\pi_s$ such that $s\le t$, $\pi_s\le \pi_t$, and $\pi_s$
  is as small as possible with these properties. This means that $\pi_s$
  was colored blue for the reason that coloring it red would have
  completed a red occurrence of $\sigma\oplus(\tau\ominus 1)$.
  Therefore, $\pi_s$ is the rightmost element of an occurrence of
  $\sigma\oplus(\tau\ominus 1)$, and all other elements of this
  occurrence are red.  Let $\sigma_R$ and $\tau_R$ be the sets of
  elements corresponding to $\sigma$ and $\tau$ in this occurrence.

  We now distinguish two cases depending on the relative position of
  $\tau_B$ and $\sigma_R$. If all the elements of $\sigma_R$ precede
  all the elements of $\tau_B$ (including the case $\tau=\emptyset$),
  then $\sigma_R\cup\tau_B\cup 1_B\cup \rho_B$ forms an occurrence of
  $\sigma\oplus(\tau\ominus 1)\oplus\rho$. This is because each
  element of $\sigma_R$ is smaller than $\pi_s$, which, in turn, is at
  most as large as~$\pi_t$.

  Suppose now that at least one element of $\sigma_R$ is to the right
  of the leftmost element of $\tau_B$. Then all the elements of
  $\tau_R$ are to the right of the leftmost element
  of~$\tau_B$. Consequently, all the elements of $\tau_R$ are smaller
  than the leftmost element of $\tau_B$, otherwise they would be
  blue. Therefore all elements of $\tau_R$ are smaller than any
  element of $\rho_B$, and $\sigma_R\cup \tau_R\cup\{\pi_s\}\cup
  \rho_B$ is an occurrence of $\sigma\oplus(\tau\ominus 1)\oplus\rho$.
\end{proof}

The special case $\tau=\emptyset$, in the lemma above, corresponds to
an argument by B\'ona~\cite{bona-records}. More importantly, for our
purposes, the special case $\sigma=\tau=\rho=1$ gives a representation 
of any 1324-avoiding permutation as a merge of a 132-avoiding
permutation and a 213-avoiding permutation. For instance, coloring the
1324-avoiding permutation $364251$ we find that the red and blue
elements are $3621$ and $45$, respectively.

To apply Lemma~\ref{lemma:red-blue}, we need an estimate on the number
of permutations obtainable by merging two permutations from given
\emph{permutation classes}.  A permutation class is a set of
permutations $\C$ that is down-closed for the containment relation,
that is, if $\tau\in \C$ and $\tau$ contains $\sigma$, then $\sigma\in
\C$. The \emph{growth rate} of $\C$ is defined as
$\limsup_{n\to\infty}|\C\cap \sym_n|^{1/n}$. As pointed out by
Arratia~\cite{Arratia}, if $\C =\sym(\tau)$ is a principal class, that
is, $\C$ is the set of permutations avoiding a single pattern~$\tau$,
then the $\limsup$ is actually a limit, and the growth rate of $\C$ is
the Stanley-Wilf limit $L(\tau)$ of~$\tau$.

In the following lemma we adapt an argument that has been used by
B\'ona~\cite{bona-records} in a less general setting.

\begin{lemma}\label{lemma:merge}
  Let $\A$, $\B$, and $\C$ be three permutation classes with growth
  rates $\alpha$, $\beta$ and $\gamma$, respectively. If every
  permutation of $\C$ can be expressed as a merge of a permutation
  from $\A$ and a permutation from $\B$, then
  $$
  \sqrt{\gamma}\le \sqrt{\alpha}+\sqrt{\beta}.
  $$
\end{lemma}
\begin{proof}
  Let $a_n$, $b_n$ and $c_n$ be the numbers of permutations of
  length~$n$ in $\A$, $\B$ and $\C$, respectively. For every
  $\varepsilon>0$, we may fix a constant $K$ such that $a_n\le
  K\alpha^n(1+\varepsilon)^n$ and $b_n\le K\beta^n(1+\varepsilon)^n$
  for each~$n$.

  There are at most $\binom{n}{k}^2$ possibilities to merge a given
  permutation of length $k$ with a given permutation of length $n-k$,
  because we get to choose $k$ positions and $k$ values to be covered
  by the first permutation. We thus have
  \begin{align*}
    c_n \le \sum_{k=0}^n \binom{n}{k}^2 a_k b_{n-k}
    &\le K^2(1+\varepsilon)^n \sum_{k=0}^n \binom{n}{k}^2 \alpha^k \beta^{n-k}\\
    &\le K^2(1+\varepsilon)^n\sum_{k=0}^n
\left(\binom{n}{k}\sqrt{\alpha}^k\sqrt{\beta}^{n-k}\right)^2\\
    &\le
K^2(1+\varepsilon)^n\left(\sum_{k=0}^n\binom{n}{k}\sqrt{\alpha}^k\sqrt{\beta}^{
n-k}\right)^2\\
    &\le K^2(1+\varepsilon)^n\left(\sqrt{\alpha}+\sqrt{\beta}\right)^{2n},
  \end{align*}
  which implies that $\gamma$ is at most
  $(\sqrt{\alpha}+\sqrt{\beta})^2$, as claimed.
\end{proof}

Taking $\sigma=\tau=\rho=1$ in Lemma~\ref{lemma:red-blue}, using the
fact that $L(132)=L(213)=4$, and applying Lemma~\ref{lemma:merge}, we
get the following result.
\begin{corollary}\label{cor:1324}
  The Stanley-Wilf limit of 1324 is at most 16.\qed
\end{corollary}

We may apply Lemmas~\ref{lemma:red-blue} and \ref{lemma:merge} to get
an upper bound for the Stanley-Wilf limit of any layered pattern. Let
$\oplus^n1$ denote the identity permutation $12\dotsb n$, and let
$\ominus^n1$ denote its reverse $n\dotsb21$. By a result of Backelin,
West and Xin~\cite{bwx}, we know that for arbitrary~$\sigma$, the
pattern $(\oplus^n1)\oplus\sigma$ is Wilf equivalent to the pattern
$(\ominus^n1)\oplus\sigma$.

Let $\alpha(\ell_1,\ell_2,\dots,\ell_m)$ denote the Stanley-Wilf limit
of the generic layered permutation $(\ominus^{\ell_1}1)\oplus
(\ominus^{\ell_2}1)\oplus\dotsb\oplus (\ominus^{\ell_m}1)$.

\begin{lemma}\label{lemma:layered}
  For any positive integers $\ell_1,\dots,\ell_m$, we have
  $$
  \alpha(\ell_1,\dots,\ell_m)\le\left(\ell_1 + \ell_m - m + 1 +2\sum_{i=2}^{m-1} \ell_i\right)^2.
  $$
\end{lemma}
\begin{proof}
  We proceed by induction on $m$. For $m\le 2$, the bound follows from
  the known fact~\cite{Regev} that $\oplus^k1$ has Stanley-Wilf limit
  $(k-1)^2$. Assume now that $m\ge 3$. Combining
  Lemma~\ref{lemma:red-blue} and Lemma~\ref{lemma:merge}, we see that
  \begin{align*}
    \sqrt{\alpha(\ell_1,\dots,\ell_m)}
    &\le \sqrt{\alpha(\ell_1,\ell_2)}+\sqrt{\alpha(\ell_2,\dots,\ell_m)} \\
    &\le (\ell_1+\ell_2-1)+\left(\ell_2 + \ell_m - m + 2 + 2\sum_{i=3}^{m-1} \ell_i\right),
  \end{align*}
  which gives the desired bound.
\end{proof}

\begin{corollary}\label{cor:layered}
  A layered permutation of length $\ell$ has Stanley-Wilf limit at
  most $4\ell^2$.
\end{corollary}
As we pointed out in the introduction, any layered pattern of length
$\ell$ has Stanley-Wilf limit at least $(\ell-1)^2$, so the quadratic
bound in the previous corollary is best possible.

\section{On 1324-avoiding permutations with a fixed number of
  inversions}
\label{sec:conjectures}

An \emph{inversion} in a permutation $\pi=\pi_1\pi_2\dotsb \pi_n$ is a
pair $(i,j)$ such that $1\le i<j\le n$ and $\pi_i>\pi_j$. The number
of inversions in $\pi$ is denoted $\inv(\pi)$. In this section we will
consider the distribution of inversions over $1324$-avoiding
permutations. We will show that a certain conjectured property of this
distribution implies an improved upper bound for $L(1324)$. Recall that 
$\snk{n}{k}(\tau)$ is the number of $\tau$-avoiding permutations of length $n$ with $k$ inversions.

To illustrate our approach, and to introduce tools we use later, we
will first derive an upper bound on $L(132)$ (even though we know that
$L(132)=4$). Here are the first few rows of the distribution of
inversions over $\sym(132)$, where the $k$th entry in the $n$th row is
the number $\snk{n}{k}(132)$:
$$
\begin{tikzpicture}
  \matrix (A) [matrix of math nodes, row sep=0.2mm, column sep=0.2mm] {
    \g{1}\\
    \g{1}&\g{1}\\
    \g{1}&\g{1}&\g{2}&1\\
    \g{1}&\g{1}&\g{2}&\g{3}&3&3&1\\
    \g{1}&\g{1}&\g{2}&\g{3}&\g{5}&5&7&7&6&4&1\\
    \g{1}&\g{1}&\g{2}&\g{3}&\g{5}&\g{7}&9&11&14&16&16&17&14&10&5&1\\
    \g{1}&\g{1}&\g{2}&\g{3}&\g{5}&\g{7}&\g{11}&13&18&22&28&32&37&40&44&43&\dots\\
    \g{1}&\g{1}&\g{2}&\g{3}&\g{5}&\g{7}&\g{11}&\g{15}&20&26&34&42&53&63&73&85&\dots\\
  };
\end{tikzpicture}
$$
We make two observations: (1) the columns are weakly increasing when
read from top to bottom; (2) each column is eventually constant, as shown by the grayed area. If we can prove this and give a formula for the eventual value $c(k)$ of the $k$-th column, then we can bound $\s_n(132)$ by $\sum_{k\le\binom{n}{2}} c(k)$. For instance, $\s_4(132) \leq
1+1+2+3+5+7+11 = 30$.

To prove that the columns are weakly increasing is easy: the map
$\pi \mapsto \pi\oplus 1$ from $\sym_{n-1}(132)$ to $\sym_n(132)$ is
injective and inversion-preserving. Our goal is to show that each column is eventually constant and to find the formula for the eventual value of $k$-th column.

\begin{lemma}\label{lemma:comp+inv}
  Let $\pi\in\sym_n$ and let $c$ be the number of components of
  $\pi$. 
  Then
  $$\inv(\pi) \geq n - c.
  $$
\end{lemma}

\begin{proof}
  We use induction on $n$. The case $n=1$ is trivial. Assume $n>1$ and
  write $\pi$ as the sum of its components $\pi = \alpha_1 \oplus
  \dots \oplus \alpha_c$.  Note that if $(i,j)$ is an inversion in
  $\pi=\pi_1\pi_2\dotsb \pi_n$ then $\pi_i$ and $\pi_j$ must belong to
  the same component of $\pi$.  Thus, if $c>1$, we have
  \begin{align*}
    \inv(\pi) 
    &= \inv(\alpha_1) +\dotsb+ \inv(\alpha_c) \\
    &\geq  |\alpha_1|-1  +\dotsb+  |\alpha_c|-1 &\text{by induction}\\
    &= n-c.
  \end{align*}
  Assume then that $c=1$ and let $i$ be the position of $n$ in $\pi$,
  that is, $\pi_i=n$. Also, let $\sigma\in\sym_{n-1}$ be the permutation
  obtained by removing $n$ from $\pi$. Although $\sigma$ can be any
  permutation, the number $i$ has some restrictions. Obviously, $i\leq
  |\sigma|=n-1$, since if $i=n$, then $n$ would constitute a component
  of its own, contradicting the assumption that $c=1$. More generally,
  if we decompose $\sigma$ into its components
  $$\sigma = \beta_1 \oplus \dots \oplus \beta_{d},
  $$
  then we see that $i\leq|\beta_1|$. Thus
  \begin{align*}
    \inv(\pi) 
    &= \inv(\sigma) + n-i \\
    &\geq |\sigma|-d + n-i &\text{by induction} \\
    &=n-1-i + n-d          &\text{since $|\sigma|=n-1$}\\
    &\geq n-1-i+|\beta_1|  &\text{since $n-d\geq |\beta_1|$}\\
    &\geq n-1,             &\text{since $i \leq |\beta_1|$}
  \end{align*}
  as claimed.
\end{proof}

\begin{definition}
  The \emph{inversion table} of a permutation $\pi=\pi_1\dotsb\pi_n$
  is the sequence $b_1b_2\dotsb b_n$, where $b_i$ is the number of
  letters in $\pi$ to the right of $\pi_i$ that are smaller than
  $\pi_i$.
\end{definition}

For example, the inversion table of 352614 is 231200.  Clearly, the
number of inversions in a permutation equals the sum of the entries in
the inversion table. It is also easy to see that the map taking a
permutation to its inversion table is a bijection.  In other words, a
permutation can be reconstructed from this table.

\begin{lemma}
  A permutation avoids the pattern 132 if and only if its inversion
  table is weakly decreasing.
\end{lemma}
\begin{proof}
  Let $\pi=\pi_1\pi_2\dotsb \pi_n$ be a permutation with
  inversion table $b_1b_2\dotsb b_n$. Suppose that $b_i<b_{i+1}$ for
  some $i$.  Then we must have $\pi_i<\pi_{i+1}$, and the number of
  letters to the right of $\pi_{i+1}$ that are smaller than $\pi_{i+1}$ is
  greater than the number of such letters that are smaller than $\pi_i$.
  Thus there is a $j>i+1$ such that $\pi_i<\pi_j<\pi_{i+1}$.  But then
  $\pi_i\pi_{i+1}\pi_j$ form the pattern 132. 

  Conversely, assume that $\pi_i\pi_j\pi_k$ is an occurrence of 132 in
  $\pi$, and assume that the occurrence has been chosen in such a way
  that the index $i$ is as large as possible. This choice implies that
  $\pi_{i+1}$ is greater than $\pi_k$, and consequently,
  $b_{i+1}>b_i$.
\end{proof}

A \emph{partition} of an integer $k$ is a weakly decreasing sequence
of positive integers whose sum is~$k$. By dropping the trailing zeros
from the inversion table of a permutation $\pi\in\symnk{n}{k}(132)$ we
obtain a partition $\lambda$ of~$k$. We then say that $\lambda$
\emph{represents}~$\pi$. For instance, the inversion table of
$\pi=65723148$ is 5441100, so $\pi$ is represented by the integer
partition $5+4+4+1+1$. Two distinct 132-avoiding permutations of the same size are represented by distinct integer partitions. On the other hand, a
permutation $\pi\in\sym_n(132)$ is represented by the same partition
as $\pi\oplus 1$. 

In any $132$-avoiding permutation $\pi$, only the
first component may have size greater than~1, so $\pi$ has a
decomposition of the form $\sigma\oplus1\oplus1\oplus\dotsb\oplus 1$
where $\sigma$ is an indecomposable permutation represented by the
same partition as~$\pi$. It is easy to see that a partition $\lambda$ of an integer $k$ represents a unique indecomposable permutation~$\sigma$, and by Lemma~\ref{lemma:comp+inv}, $\sigma$ has size at most~$k+1$.
Consequently, for every $n\ge k+1$, $\lambda$ represents a
unique permutation $\pi$ of size~$n$. This yields the following
result.

\begin{proposition}\label{prop:132}
  For every $k<n$, we have $\snk{n}{k}(132) = p(k)$, where $p(k)$ is
  the number of integer partitions of $k$.
\end{proposition}

The following rather elementary upper bound for $p(k)$ can, for
example, be found in ~\cite{apostol}, pp 316--318.

\begin{lemma}\label{lemma:partitions}
  Let $p(k)$ be the number of integer partitions of $k$. For $k>0$ we have
  $$ p(k) <  \rho^{\sqrt{k}},
  $$
  where $\rho=e^{\pi\sqrt{\frac{2}{3}}}\simeq 13.001954$.
\end{lemma}

Letting $m=\binom{n}{2}$, we thus have
\begin{align*}
  \s_n(132) =\sum_{k=0}^{m} \snk{n}{k}(132)
  &\leq (m+1) \snk{m+1}{m} \\[-2ex]
  & = (m+1) p(m+1) < (m+1) e^{\pi\sqrt{\frac{2(m+1)}{3}}},
\end{align*}
and
\begin{align*}
  L(132) = \lim_{n\to\infty} (\s_n(132))^{1/n}
  & \leq \lim_{n\to\infty}\!\left(n^2/2+O(n)\!\right)^{\! 1/n}e^{\frac{\pi}{n}\sqrt{\frac{2}{3}\left(\frac{n^2}{2}+O(n)\!\right)}}\\
  &= e^{\frac{\pi}{\sqrt{3}}} \simeq 6.1337.
\end{align*}

We have now seen that counting 132-avoiding permutations with
relatively few inversions leads to an upper bound for $L(132)$. Can we
similarly bound $L(1324)$? These are the first few rows of the
distribution of inversion over $\sym(1324)$:
$$
\begin{tikzpicture}
  \matrix (A) [matrix of math nodes, row sep=0.2mm, column sep=0.2mm] {
    1\\
    \g{1}& 1\\ 
    \g{1}& \g{2}& 2& 1\\
    \g{1}& \g{2}& \g{5}& 6& 5& 3& 1\\
    \g{1}& \g{2}& \g{5}& \g{10}& 16& 20& 20&  15&   9&   4&   1\\
    \g{1}& \g{2}& \g{5}& \g{10}& \g{20}& 32& 51&  67&  79&  80&  68&  49&   29& \dots\\
    \g{1}& \g{2}& \g{5}& \g{10}& \g{20}& \g{36}& 61&  96& 148& 208& 268& 321&  351& \dots\\
    \g{1}& \g{2}& \g{5}& \g{10}& \g{20}& \g{36}& \g{65}& 106& 171& 262& 397& 568&  784& \dots\\
    \g{1}& \g{2}& \g{5}& \g{10}& \g{20}& \g{36}& \g{65}& \g{110}& 181& 286& 443& 664&  985& \dots\\
    \g{1}& \g{2}& \g{5}& \g{10}& \g{20}& \g{36}& \g{65}& \g{110}& \g{185}& 296& 467& 714& 1077& \dots\\
    \g{1}& \g{2}& \g{5}& \g{10}& \g{20}& \g{36}& \g{65}& \g{110}& \g{185}& \g{300}& 477& 738& 1127& \dots\\
    \g{1}& \g{2}& \g{5}& \g{10}& \g{20}& \g{36}& \g{65}& \g{110}& \g{185}& \g{300}& \g{481}& 748& 1151& \dots\\
  };
\end{tikzpicture}
$$
Again, it seems as though (1) the columns are increasing when read
from top to bottom; (2) each column is eventually constant. Unfortunately, we have not been able to show that the columns are increasing, but we conjecture that they are:

\begin{conjecture}[Increasing columns]\label{increasing-columns}
  For all nonnegative integers $n$ and $k$, we have
  $\snk{n}{k}(1324)\le\snk{n+1}{k}(1324)$.
\end{conjecture}

We are, however, able to say what the fixed sequence is. First we
need some lemmas. 

\begin{lemma}\label{lemma:132+213}
  If $\pi\in\snk{n}{k}(1324)$ and $k < n-1$, then
  $$
  \pi = \sigma \oplus 1 \oplus \dots \oplus 1 \oplus \tau
  $$
  for some nonempty permutations $\sigma\in\sym(132)$ and
  $\tau\in\sym(213)$.
\end{lemma}

\begin{proof}
  Let $c$ be the number of components in $\pi$, and write $\pi = \alpha_1
  \oplus \dotsb \oplus \alpha_c$. Since $k<n-1$ it follows from
  Lemma~\ref{lemma:comp+inv} that $c\geq 2$. There can be no
  inversions in $\pi$ except within the first component and within
  the last component, since otherwise we would have an occurrence of
  $1324$; thus $\alpha_2=\cdots =\alpha_{c-1}=1$. Since the letters in the
  first component have a larger letter to their right, the first
  component must avoid 132 (so that $\pi$ avoids 1324). Likewise, the
  last component must avoid 213.
\end{proof}

Let $\P(m)$ be the set of partitions of an integer~$m$, and let $\Q(m)$ be the
set
\[
\{\, (\lambda,\mu):\, \lambda\in \P(i),\, \mu\in \P(j),\, i+j=m \,\}. 
\]

\begin{proposition}\label{prop:small-inv} 
  For $k < n-1$, there is a one-to-one correspondence between
  $\symnk{n}{k}(1324)$ and $\Q(k)$.
\end{proposition}

\begin{proof} Fix $k<n-1$ and choose $\pi\in\symnk{n}{k}(1324)$.  Let
  $c$ be the number of components in $\pi$, and write $\pi = \alpha_1
  \oplus\alpha_2 \oplus \dots \oplus \alpha_c$, where each $\alpha_i$
  is nonempty. By Lemma~\ref{lemma:132+213} we have that
  $\alpha_1\in\sym(132)$, $\alpha_c\in\sym(213)$, and
  $\alpha_2=\cdots=\alpha_{c-1}=1$. Let $\sigma=\alpha_1$ and
  $\tau=\alpha_c$. Let $i=\inv(\sigma)$, and $j=\inv(\tau)$. Note that
  $i+j=k$.

  Let $\ell$ be the length of $\tau$ and let $\tau'$ be the
  reverse-complement of $\tau$, that is, $\tau'_i= \ell+i -
  \tau_{\ell+1-i}$ for $i=1,\dotsc,\ell$. Then $\tau'$ is an
  indecomposable 132-avoiding partition with $j$ inversions. Let
  $\lambda\in\P(i)$ and $\mu\in\P(j)$ be the partitions representing
  $\sigma$ and $\tau'$, respectively. We then let
  $(\lambda,\mu)\in\Q(k)$ be the image of~$\pi$.

  To see that this is a bijection, choose a pair
  $(\lambda,\mu)\in\Q(k)$. Let $i$ be the sum of $\lambda$ and $j$ the
  sum of~$\mu$. Let $\sigma$ and $\tau'$ be the unique indecomposable
  132-avoiding permutations represented by $\lambda$ and $\mu$,
  respectively. Let $\tau$ be the reverse-complement of~$\tau'$. We
  have $n\ge k+2 = (i+1)+(j+1)\ge |\sigma|+|\tau|$. We can therefore
  construct a permutation $\pi=\sigma\oplus
  (\oplus^{n-|\sigma|-|\tau|} 1) \oplus\tau$, which is the preimage of
  $(\lambda,\mu)$.
\end{proof}

\begin{lemma}\label{lemma:pairs-of-partitions}
  Let $\rho$ be as in Lemma~\ref{lemma:partitions}. Then, for $k>0$,
  $$|\Q(k)| < (k+1)\rho^{\sqrt{2k}}.
  $$
\end{lemma}

\begin{proof}
  We have
  \begin{align*}
    |\Q(k)| = \sum_{i=0}^{k}p(i)p(k-i)
    &< \sum_{i=0}^{k}\rho^{\sqrt{i}+\sqrt{k-i}}
    & \text{by Lemma~\ref{lemma:partitions}} \\
    &\leq \sum_{i=0}^{k}\rho^{\sqrt{2k}}
    & \text{since } \sqrt{i} + \sqrt{k-i} \leq \sqrt{2k} \\
    &= (k+1)\rho^{\sqrt{2k}},
  \end{align*}
  as claimed
\end{proof}

\begin{theorem}
  If Conjecture~\ref{increasing-columns} is true, then the Stanley-Wilf limit
for $1324$ is at most $\rho=e^{\pi\sqrt{\frac{2}{3}}}\simeq 13.001954$.
\end{theorem}

\begin{proof}
  With $m=\binom{n}{2}$ we have
  \begin{align*}
    \s_n(1324) = \sum_{k=0}^{m} \snk{n}{k}(1324)
    &\leq \sum_{k=0}^{m} \snk{m+2}{k}(1324)
    & \text{by Conjecture~\ref{increasing-columns}} \\
    &=  \sum_{k=0}^{m} |\Q(k)|
    & \text{by Proposition~\ref{prop:small-inv}} \\
    &< \sum_{k=0}^{m} (k+1)\rho^{\sqrt{2k}}
    & \text{by Lemma~\ref{lemma:pairs-of-partitions}} \\
    &\leq (m+1)(m+1)\rho^{\sqrt{2m}}\\
    &= \tfrac{1}{4}(n^2 - n + 2)^2
       \rho^{n\sqrt{1 - 1/n}}.
  \end{align*}
  On taking the $n$th root and letting $n\to\infty$, the result follows.
\end{proof}

\section{Generalizations}
\label{sec:general}

We have seen that the conjectured inequality $\snk{n}{k}(1324)\le
\snk{n+1}{k}(1324)$ implies an estimate on $L(1324)$. Let us now focus
on the behavior of $\snk{n}{k}(\tau)$ for general patterns $\tau$. Let
us say that a pattern $\tau$ is \emph{$\inv$-monotone} if for every
$n$ and every $k$, we have the inequality
$\snk{n}{k}(\tau)\le\snk{n+1}{k}(\tau)$.

Recall that $\oplus^\ell 1$ is the identity pattern
$12\dotsb\ell$. Let us first observe that a pattern of this form
cannot be $\inv$-monotone.
\begin{lemma}\label{obs-increasing}
  For any $k$ and for any $n$ large enough, we have
  $\snk{n}{k}(\oplus^\ell 1)=0$. In particular, $\oplus^\ell 1$ is not
  $\inv$-monotone for any $\ell\ge 2$.
\end{lemma}
\begin{proof}
  For $n > (\ell-1)(k+1)$, a well-known result of Erd\H os and
  Szekeres~\cite{ErdosSzekeres} guarantees that any permutation of
  length $n$ has either an increasing subsequence of length $\ell$ or
  a decreasing subsequence of length $k+2$. Therefore, there can be no
  $\oplus^\ell 1$-avoiding permutations of length $n$ with $k$
  inversions.
\end{proof}

On the other hand, some patterns are $\inv$-monotone for trivial reasons,
as shown by the next lemma.

\begin{lemma}
  Let $\tau=\tau_1\tau_2\dotsb\tau_\ell$ be a pattern such that
  $\tau_1> 1$ or $\tau_\ell<\ell$. Then $\tau$ is $\inv$-monotone.
\end{lemma}
\begin{proof}
  Suppose that $\tau_1> 1$. It is plain that $\pi\mapsto 1\oplus\pi$
  is an injection from $\symnk{n}{k}(\tau)$ into
  $\symnk{n+1}{k}(\tau)$, demonstrating that
  $\snk{n}{k}(\tau)\le\snk{n+1}{k}(\tau)$. The other case is
  symmetric.
\end{proof}

We are not able to characterize the $\inv$-monotone patterns. Based on
numerical evidence obtained for patterns of small size, we make the
following conjecture, which generalizes
Conjecture~\ref{increasing-columns}.

\begin{conjecture}\label{general-increasing-columns}
  Any pattern $\tau$ that is not an identity pattern is $\inv$-monotone.
\end{conjecture}

Another source of support for
Conjecture~\ref{general-increasing-columns} comes from our analysis of
the asymptotic behavior of $\snk{n}{k}(\tau)$ as $n$ tends to
infinity. To state the results precisely, we need some definitions.

A \emph{Fibonacci permutation} is a permutation $\pi$ that can be
written as a direct sum
$\pi=\alpha_1\oplus\alpha_2\oplus\dotsb\oplus\alpha_m$ where each
$\alpha_i$ is equal to 1 or to 21. In other words, a Fibonacci
permutation is a layered permutation whose every layer has size at
most~2.

\begin{proposition}\label{pro-fibonacci}
  Let $\tau$ be a Fibonacci pattern with $r\ge 1$ inversions. For
  every $k\ge r$, there is a polynomial $P$ of degree $r-1$ and an
  integer $n_0\equiv n_0(k,\tau)$, such that $\snk{n}{k}(\tau)=P(n)$
  for all $n\ge n_0$.
\end{proposition}
\begin{proof}
  We first observe that an arbitrary permutation $\pi$ can be uniquely
  expressed as a direct sum (possibly involving a single summand) of
  the form
  \[
  \pi=\alpha_0\oplus\beta_1\oplus\alpha_1\oplus\beta_2\oplus\dotsb\oplus\beta_m
  \oplus\alpha_m,
  \]
  where $m\ge 0$ is an integer, each $\alpha_i$ is a (possibly empty)
  identity permutation, and each $\beta_i$ is an indecomposable
  permutation of size at least two. For instance, if $\pi=124365$, we
  have $\alpha_0=12$, $\alpha_1=\alpha_2=\emptyset$, and
  $\beta_1=\beta_2=21$.  We will call the sequence
  $(\beta_1,\dotsc,\beta_m)$ \emph{the core} of $\pi$, and the
  sequence $(\alpha_0,\dotsc,\alpha_m)$ \emph{the padding}
  of~$\pi$. The sequence of integers $(|\alpha_0|,\dotsc,|\alpha_m|)$
  will be referred to as the \emph{padding profile} of~$\pi$. Of
  course, the padding is uniquely determined by its profile, and the
  permutation $\pi$ is uniquely determined by its core and its padding
  profile.

  Let $\tau$ be a Fibonacci pattern with $r$ inversions. Note that
  this is equivalent to saying that $\tau$ is a permutation whose core
  consists of $r$ copies of~$21$. Let $\ell$ be the length
  of~$\tau$. Let us fix an integer $k\ge r$ and focus on the values of
  $\snk{n}{k}(\tau)$ as a function of~$n$.

  Note that the core of a permutation $\pi$ with $\inv(\pi)=k$ can
  have at most $k$ components. Moreover, each component of the core of
  $\pi$ has size at most $k+1$, otherwise $\pi$ would have more than
  $k$ inversions by Lemma~\ref{lemma:comp+inv}. In particular, the
  permutations with $k$ inversions have only a finite number of
  distinct cores. Define $\sym^{k}(\tau)=\bigcup_{n\ge 1}
  \symnk{n}{k}(\tau)$. Let $C$ be the set of all the distinct cores
  formed by members of $\sym^{k}(\tau)$. Let $\symnk{n}{[c]}(\tau)$ be
  the set of permutations from $\symnk{n}{k}(\tau)$ whose core is
  equal to~$c$, and let $\snk{n}{[c]}(\tau)$ be its cardinality.
  Clearly, $\snk{n}{k}(\tau)=\sum_{c\in C}\snk{n}{[c]}(\tau)$.

  To prove our proposition, it is enough to prove the following three
  claims:
  \begin{enumerate}
  \item There is a constant $\gamma\equiv \gamma(k,\tau)$ such that
    $\snk{n}{k}(\tau)\ge \gamma n^{r-1}$ for every~$n$.
  \item For every $c\in C$, there is a constant
    $\delta\equiv\delta(k,c,\tau)$ such that $\snk{n}{[c]}(\tau)\le
    \delta n^{r-1}$.
  \item For every $c\in C$, there is a polynomial $P_c$ and a constant
    $n_c$ such that $\snk{n}{[c]}(\tau)=P_c(n)$ for every $n\ge n_c$.
  \end{enumerate}

  To prove the first claim, we consider the set of all permutations
  with core $c=(\beta_1,\beta_2,\dotsc,\beta_{r-1})$, where
  $\beta_1=1\ominus(\oplus^{k-r+2}1)=(k-r+3)12\dotsb(k-r+2)$, and
  $\beta_2=\beta_3=\dotsb=\beta_{r-1}=21$. Note that any permutation
  with core $c$ has exactly $k$ inversions and avoids $\tau$. The
  padding profile of such a permutation is a sequence of $r$
  nonnegative numbers whose sum is $n-\sum_{i=1}^{r-1}|\beta_i| =
  n-k-r+1$. The number of such sequences is $\binom{n-k}{r-1}$, which
  gives the claimed bound.

  To prove the second claim, fix a core $c=(\beta_1,\dotsc,\beta_m)\in
  C$, and define $t=\sum_{i=1}^m|\beta_i|$. Assume that $m\ge r$,
  otherwise there are only $O(n^{r-1})$ permutations with core $c$ and
  the claim is trivial. Let $\pi$ be a permutation with core~$c$, and
  let $a=(a_0,\dotsc,a_m)$ be the padding profile of~$\pi$. Observe
  that if $a$ has more than $r$ integers greater than $\ell$, then
  $\pi$ must contain~$\tau$. Thus, $\snk{n}{[c]}(\tau)$ can be bounded
  from above by the number of all the padding profiles of sum $n-t$
  and with at most $r$ components greater than~$\ell$. The number of
  such padding profiles may be bounded from above by
  $\binom{m+1}{r}\ell^{m+1-r}\binom{n-t}{r-1}$, proving the second
  claim.

  To prove the last claim, we reduce it to a known property of
  down-sets of integer compositions. Let $\bbN_0^d$ be the set of $d$-tuples of
  non-negative integers. Fix a core
  $c=(\beta_1,\dotsc,\beta_m)\in C$.  Let $a(\pi)$ denote the
  padding profile of a permutation~$\pi$. Define the sets $A_n=\{a(\pi):
  \pi\in\symnk{n}{[c]}\}$ and $A=\bigcup_{n\ge 0} A_n$. Define a
  partial order $\le$ on $\bbN_0^{m+1}$ by putting $(a_0,\dotsc,a_m)\le
  (b_0,\dotsc,b_m)$ if for every $i\in\{0,\dotsc,m\}$ we have $a_i\le
  b_i$. Note that for two permutations $\sigma$ and $\pi$ with core
  $c$, $\sigma$ is contained in $\pi$ if and only if $a(\sigma)\le
  a(\pi)$. In particular, the set $A$ is a down-set of $\bbN_0^{m+1}$,
  that is, if $a$ belongs to $A$ and $b\le a$, then $b$ belongs to $A$
  as well. To complete the proof, we use the following fact, due to
  Stanley~\cite{stanley,stanley2}.
  \begin{proposition}[Stanley]
    For every $d$, if $D$ is a down-set in $\bbN_0^d$ and $D(n)$ is the
    cardinality of the set $\{(a_1,\dotsc,a_d)\in D:
    a_1+\dotsb+a_d=n\}$, then there is a polynomial $P$ such that
    $D(n)=P(n)$ for all $n$ sufficiently large.
  \end{proposition}
  From this fact, we directly obtain that $|A_n|$ is eventually equal
  to a polynomial, and therefore $\snk{n}{[c]}(\tau)$ is eventually
  equal to a polynomial as well.
\end{proof}

Let $P(n)$ be the polynomial from Proposition~\ref{pro-fibonacci}. We
note that if $\tau = \oplus^{\ell}1$, then $P(n)$ is the zero polynomial
by Lemma~\ref{obs-increasing}; if $\tau = 132$, then $P(n)=p(k)$ by
Proposition~\ref{prop:132}; and if $\tau = 1324$, then $P(n)=|\Q(k)|$
by Proposition~\ref{prop:small-inv}. It would be interesting to know
what $P(n)$ is for other Fibonacci patterns.

The conclusion of Proposition~\ref{pro-fibonacci} cannot be extended
to non-Fibonacci patterns, as shown by the next proposition.

\begin{proposition}\label{pro-nonfibonacci}
  Let $\tau$ be a non-Fibonacci permutation. For every $k$ there
  exists a polynomial $P$ of degree $k$ and an integer $n_0$ such that
  for every $n\ge n_0$, $\snk{n}{k}(\tau)=P(n)$. Moreover,
  $P(n)=n^k/k!+ O(n^{k-1})$.
\end{proposition}
\begin{proof}
  We can show that $\snk{n}{k}(\tau)$ is eventually equal to a
  polynomial $P$ by the same argument as we used in the proof of
  Proposition~\ref{pro-fibonacci}. It is therefore enough to provide
  upper and lower bounds for $\snk{n}{k}(\tau)$ of the form
  $\frac{n^k}{k!}+ O(n^{k-1})$. To get the upper bound, note that the
  number of all permutations of length $n$ with $k$ inversions is at
  most $\binom{n+k-1}{k}$, as seen by encoding a permutation by its
  inversion table. For the lower bound, note that $\tau$ is not
  contained in any Fibonacci permutation, and the number of Fibonacci
  permutations of length $n$ with $k$ inversions is
  precisely~$\binom{n-k}{k}$.
\end{proof}

Propositions~\ref{pro-fibonacci} and \ref{pro-nonfibonacci} imply that
for any pattern $\tau$, any $k$ and any $n$ large enough, we have
$\snk{n}{k}(\tau)\le \snk{n+1}{k}(\tau)$, which corresponds to an
`asymptotic version' of
Conjecture~\ref{general-increasing-columns}. The two propositions also
imply a sharp dichotomy between Fibonacci and non-Fibonacci patterns,
in the sense of the next corollary.
\begin{corollary}
  Let $\snk{n}{k}$ be the number of all permutations of size $n$
  with $k$ inversions. Let $\tau$ be a pattern with $r$
  inversions. Define
  \[
  Q(k,\tau)=\lim_{n\to\infty}\frac{\snk{n}{k}(\tau)}{\snk{n}{k}}
  \]
  as the asymptotic probability that a large permutation with $k$
  inversions avoids~$\tau$. If $\tau$ is a Fibonacci pattern and $k\ge
  r$, then $Q(k,\tau)=0$.  In all other cases $Q(k,\tau)=1$.
\end{corollary}


\end{document}